\documentclass[11pt,a4paper]{article}
\usepackage[english]{babel}       \usepackage[latin1]{inputenc}
\usepackage[T1]{fontenc}          \usepackage{a4}
\usepackage[cyr]{aeguill}         \usepackage{epsfig}
\usepackage{amsmath,amsthm,amsfonts,amssymb,mathrsfs}
\usepackage{dsfont}                \usepackage{float}
\usepackage{fancyhdr}              \usepackage{varioref}
\usepackage{color}
\usepackage{url}                   \usepackage{calc}
\usepackage{wrapfig}               \usepackage{graphicx}
\usepackage{enumitem,hyperref}
\newtheorem{theorem}{Theorem} %\labelformat{theo}{Theorem~#1}
\newtheorem{definition}{Definition}[section]%\labelformat{defi}{Definition~#1}
\newtheorem{remark}{Remark}[section]
\newtheorem{proposition}{Proposition}[section]
\newtheorem{lemma}{Lemma}[section]
\newtheorem{corollary}{Corollary}[section]
\newtheorem*{conjecture}{Conjecture}
\numberwithin{equation}{section}

\usepackage{authblk}

\makeatletter
\def\namedlabel#1#2{\begingroup
    #2%
    \def\@currentlabel{#2}%
    \phantomsection\label{#1}\endgroup
}
\makeatother

\newcommand{\ensemblenombre}[1]{\mathbb{#1}}
 
\newcommand{\Z}{\ensemblenombre{Z}} 
\newcommand{\R}{\ensemblenombre{R}}
\newcommand{\A}{\mathcal{A}}
\newcommand{\StateSpace}{S}
\newcommand{\ConfSpace}{\Omega}
\newcommand{\MultiConfSpace}{\boldsymbol{\Omega}}
\newcommand{\Conf}{\omega}
\newcommand{\MultiConf}{\boldsymbol{\omega}}
\newcommand{\Dim}{d}
\newcommand{\Types}{q}
\newcommand{\Typesbis}{\bar{\Types}}

\newcommand{\SigmaAlgebra}{\mathcal{F}}
\newcommand{\MultiSigmaAlgebra}{\boldsymbol{\mathcal{F}}}
\newcommand{\Leb}{\mathcal{L}^\Dim}

\newcommand{\Poisson}[1]{\pi^{#1}}
\newcommand{\MultiPoisson}[1]{\boldsymbol{\pi}^{#1}}
\newcommand{\Radius}{Q}
\newcommand{\MultiRadius}{\boldsymbol{\Radius}}
\newcommand{\Intensity}{z}
\newcommand{\MultiIntensity}{\boldsymbol{\Intensity}}
\newcommand{\MultiParamPoisson}{\MultiIntensity,\MultiRadius}

\newcommand{\MultiP}{\boldsymbol{P}}
\newcommand{\PartitionFunction}[1]{\boldsymbol{Z}(#1)}
\newcommand{\zr}{\boldsymbol{Z}_{r_1}(\Lambda, \MultiConf_{\Lambda^c})}
\newcommand{\FinitePartitionFunction}{\boldsymbol{Z}_{n}}

\newcommand{\Entropy}{\mathcal{I}}

\newcommand{\PercolationThreshold}{z_c}

\newcommand{\Specification}[1]{\mathscr{P}^{wr}_{#1}}

\newcommand{\Proportion}{\alpha}
\newcommand{\MultiProportion}{\boldsymbol{\alpha}}

\newcommand{\ncc}{N_{cc}}
\newcommand{\nccd}{N_{cc}^{r}}

\newcommand{\tP}{\widetilde{P}}

\newcommand{\1}{\mathds{1}}

\begin{document}
\title{Phase transition for continuum Widom-Rowlinson model with random radii}
\author[1]{David Dereudre}
\author[2]{Pierre Houdebert}
\affil[1]{Laboratoire de Math\'ematiques Paul Painlev\'e\\University of Lille 1, France 
 \texttt{david.dereudre@math.univ-lille1.fr}}
\affil[2]{Aix Marseille Univ, CNRS, Centrale Marseille, I2M, Marseille, France 
 \texttt{pierre.houdebert@gmail.com}}
\maketitle
%%%%%%%%%%%%%%%%%%%%%%%%%%%%%%%%%

\begin{abstract}
{In this paper we study the phase transition of continuum Widom-Rowlinson measures in $\R^d$ with $\Types$ types of particles and random radii. 
Each particle $x_i$ of type $i$ is marked by a random radius $r_i$ distributed by a probability measure $\Radius_i$ on $\R^+$. The distributions $Q_i$ may be different for different $i$, this setting is called the non-symmetric case. The particles of same type do not interact with each other whereas a particle $x_i$ and $x_j$ with different type $i\neq j$ interact via an exclusion hardcore interaction forcing $r_i+r_j$ to be smaller than $|x_i-x_j|$. 
In the symmetric integrable case (i.e. $\int r^\Dim \Radius_1(dr)<+\infty$ and $Q_i=Q_1$ for every $1\le i\le\Types$), we show that the Widom-Rowlinson measures exhibit a standard phase transition providing uniqueness, when the activity is small, and co-existence of $q$ ordered phases, when the activity is large. 
In the non-integrable case (i.e.  $\int r^\Dim \Radius_i(dr)=+\infty$, $1\le i \le q$), we show another type of phase transition. We prove, when the activity is small, the existence of at least $q+1$ extremal phases and we conjecture that, when the activity is large, only the $q$ ordered phases subsist. 
We prove a weak version of this conjecture in the symmetric case by showing that the Widom-Rowlinson measure with free boundary condition is a mixing of the $\Types$ ordered phases if and only if the activity is large. 

} 
  \bigskip

\end{abstract}

\section{Introduction} \label{section_introduction}
In this paper we deal with the non-symmetric continuum Widom-Rowlinson model in $\R^d$ with $\Types$ types of particles and with random radii. 
Each type of particle $1\le i\le \Types$ has its proper activity parameter $\Intensity_i>0$ and a proper  probability measure $\Radius_i$ on $\R^+$ for the distribution of radii (the distributions $Q_i$ may be different, this setting is called the non-symmetric case).  
Each particle $x_i$ of type $i$ is marked by a random radius $r_i$ distributed by $\Radius_i$. 
The particles of the same type do not interact with each other whereas a particle $x_i$ and $x_j$ with different type $i\neq j$ interact via an exclusion hardcore interaction forcing $r_i+r_j$ to be smaller than $|x_i-x_j|$. 
This model can be viewed as a collection of $\Types$  Boolean models, each of intensity $\Intensity_i$ and radii distribution $\Radius_i$, $i=1\dots \Types$, conditioned to not overlap each other. 

This model is a generalisation of the simple and beautiful model introduced in the late 1960' by Widom and Rowlinson \cite{widom_rowlinson} where $q=2$ and the radii are deterministic.
Its interest comes not only from its applicability in the description of a binary gas, but also from the fact that it was the very first continuum model for which a phase transition was rigorously proved, first by Ruelle \cite{ruelle_1971} using the so-called Peierls' argument. 
A modern proof of this phase transition, relying  on percolation arguments and a Fortuin-Kasteleyn representation, was done in \cite{chayes_kotecky,georgii_haggstrom}.
Regarding the non-symmetric case with $q\ge 3$, phase transition results were proved in several articles such as \cite{bricmont_kuroda_lebowitz_1984,mazel_suhov_stuhl} using the Pirogov-Sinai theory. All these results concern the case of deterministic radii.

In the present paper, we investigate the random radii case which can be interpreted as a random media or as a size distribution of particles. 
We prove  several phase transition results described below which are, depending on the distribution of radii,  similar or different from the deterministic case.

The formal definition of Widom-Rowlinson measures is based on the standard DLR equations, which prescribe the local conditional distributions of the model, see definition \ref{def_WR}. 
The existence and uniqueness of such solutions are not obvious. The set of solutions is denoted by $WR(\MultiParamPoisson)$ where $\MultiIntensity=(\Intensity_1, \dots, \Intensity_\Types)$ is the vector of activities and 
$\MultiRadius := (\Radius_1,\dots ,\Radius_\Types)$ the vector of radii distributions. 
When the radii are uniformly bounded, a general existence result by Ruelle \cite{ruelle_livre_1969} ensures the existence of Widom-Rowlinson measures. 
When the radii are not bounded, a long range interaction occurs and the existence is more delicate. 
In a first theorem we prove the existence of Widom-Rowlinson measures without any assumption on activities or radii. 
The set $WR(\MultiParamPoisson)$  is never empty.

We say that a phase transition occurs when the geometry of $WR(\MultiParamPoisson)$ changes drastically with the choice of parameters $\MultiIntensity$, considering that $\MultiRadius$ is fixed. 
In the case of deterministic radii, it is proved in papers mentioned above that $WR(\MultiParamPoisson)$ is a singleton for $\MultiIntensity$ small enough and that there exists $\MultiIntensity$ large enough such that $WR(\MultiParamPoisson)$ contains  $q$  distinct extremal ordered phases (Widom-Rowlinson measure with boundary condition full of particles with a prescribed type). 
Precisely in \cite{bricmont_kuroda_lebowitz_1984} it is proved that for any $(z_1,z_2,\ldots, z_{q-1})$ large enough there exists $z_q$ such that $WR(\MultiParamPoisson)$ contains  $q$ distinct extremal ordered phases. 
This result is based on an extension of the Pirogov-Sinai theory of phase transitions in general lattice spin systems to continuum systems. In our random radii setting we do not obtain such a general result. Actually the unbounded radii seem to be a serious and difficult obstacle for using Pirogov-Sinai machinery. Nevertheless, in the integrable setting, $\int r^\Dim \Radius_i(dr)<+\infty$, $1\le i \le q$), we show first that $WR(\MultiParamPoisson)$ is a singleton for small activities $\MultiIntensity$. 
And,  in the symmetric integrable case $\Radius_i=\Radius$, $1\le i\le q$, we show that for activities $\MultiIntensity=(z,z,\ldots,z)$ large enough, the set $WR(\MultiParamPoisson)$ contains  $q$ distinct extremal ordered phases. These results derive from a coupling result in \cite{hofer-temmel_houdebert_2017}, a standard Fortuin-Kasteleyn representation and a percolation result developed in \cite{houdebert_2017}. 
In conclusion, in the integrable symmetric setting, $WR(\MultiParamPoisson)$ exhibits a standard phase transition similar to the deterministic radii case.

Let us now turn to the most interesting and surprising result of the present paper. In the non-integrable case (i.e. $\int r^\Dim \Radius_i(dr)=+\infty$, $1\le i \le \Types$), we show another type of phase transition on the geometry of $WR(\MultiParamPoisson)$. 
First, it is easy to see that $WR(\MultiParamPoisson)$ contains always $q$ ordered phases corresponding each to a Poisson point process with only one type of particle whose  balls cover the full space $\R^\Dim$. 
But we prove, when the activity is small enough, the existence of a $\Types+1$-th extremal phase. 
As far as we know, the existence of such a $\Types+1$-th extremal phase has never been observed for a continuum Widom-Rowlinson model with deterministic radii. 
Let us note that our result is valid in the non-symmetric setting and that the proof is not based on the Pirogov-Sinai theory. 
The main ingredient is a discrimination by specific entropy. 
We show that the Widom-Rowlinson measure with free boundary condition has a specific entropy smaller than any ordered phase.

Moreover we conjecture that in the non integrable setting, when the activity is large enough, the set $WR(\MultiParamPoisson)$ is exactly the convex hull of the $\Types$ ordered phases (i.e. for large activities, only  the $\Types$ ordered phases subsist and the disordered phase disappears). It would imply a phase transition result since the set $ WR( \MultiParamPoisson)$ would have exactly $\Types$ extremal Gibbs measures for $\MultiIntensity$ large and at  least  $\Types +1$ extremal Gibbs measures for $\MultiIntensity$ small. Our belief in this conjecture is based on a similar conjecture for the continuum random cluster presented in \cite{dereudre_houdebert} and for which a heuristic proof is given. Moreover simulations in the sense of the conjecture have been implemented in \cite{houdebert_these}.

In the present paper we prove a weak version of this conjecture by showing that the symmetric Widom-Rowlinson measure  with free boundary condition is a mixing of the $\Types$ ordered phases if and only if the activity is large. 
The proof is based on a renewal argument in a prescribed direction which implies that the scales in the thermodynamic limit are different for each direction. Indeed in one direction the size of the box is of order $n$ and in the $d-1$ other directions the size is of order $\log(n)$. We believe that the result remains true with other choice of scales, in particular the standard scale where the size of the box is of the same order in each direction.

The paper is organized as follows.
Section \ref{section_preliminaries} introduces the notations, main definitions and tools. In Section \ref{section_results}   the results of the article are stated.
The proof of Theorem \ref{theo_existence_wr} concerning the existence of Widom-Rowlinson measures is done in Section \ref{section_preuve_theo_existence}.
In Section \ref{section_preuve_non_monochromaticite_petite activite} we prove the existence of a $(\Types +1)$-th extremal phase in the non-integrable setting.
Finally Section \ref{section_preuve_conjecture_faible} is devoted to the proof of the weak version of the conjecture.

\section{Preliminaries} \label{section_preliminaries}
\subsection{Space}
Let us consider the state space $\StateSpace := \R ^\Dim \times \R^+$ with $\Dim \geq 1$ being the dimension.
Let $\ConfSpace$ be the set of locally finite configurations $\Conf$ on $\StateSpace$.
This means that $|\Conf \cap (\Lambda \times \R^+)|< \infty$ for every bounded Borel set $\Lambda$ of $\R^\Dim$, with $|\omega|$ being the cardinality of the configuration $\Conf$.
We write $\Conf_{\Lambda}$ as a shorthand for $\Conf \cap (\Lambda \times \R^+)$.
The configuration space is embedded with the usual $\sigma$-algebra $\SigmaAlgebra$ generated by the counting variables.
To a configuration $\Conf \in \ConfSpace$ we associate the germ-grain structure
\begin{align*}
L(\Conf) :=
\underset{X \in \Conf}{\bigcup} B(X),
\end{align*}
where $B(X)$ is the closed ball associated to the marked point $X=(x,r)$, centred at $x$ and of radius $r$.

Let $\Types$ be an integer larger than 1 fixed through the paper, and consider the space $\MultiConfSpace := \ConfSpace^\Types$ of multi-index configurations $\MultiConf := (\Conf^1, \dots , \Conf^\Types)$ embedded with the $\sigma$-algebra 
$\MultiSigmaAlgebra := \SigmaAlgebra ^{\otimes \Types}$.
An element $i\in \{1, \dots , \Types\}$ is called a \emph{type} or a \emph{colour}.
We write $L(\MultiConf)$ as a shorthand for 
$\underset{1 \leq i \leq \Types}{\cup} L(\Conf^i)$ and $\MultiConf_{\Lambda}$
for $(\Conf^1_{\Lambda}, \dots , \Conf^\Types_{\Lambda})$.
We also write $(x,r) \in \MultiConf$ when there exists a colour $i$ such that 
$(x,r) \in \Conf^i$.

\subsection{Poisson point processes}
For $\Intensity>0$ and $\Radius$ a probability measure on $\R^+$, let 
$\Poisson{\Intensity, \Radius}$ be the distribution on $\ConfSpace$ of a Poisson point process with intensity measure $\Intensity \Leb \otimes \Radius$, where $\Leb$ stands for the Lebesgue measure on $\R^\Dim$. 
Recall that it means 
\begin{itemize}
\item for every bounded Borel set $\Lambda$, the distribution of the number of points in $\Lambda \times \R^+$ under $\Poisson{\Intensity, \Radius}$ is a Poisson distribution with parameter $z\Leb(\Lambda)$;
\item given the number of points in every bounded $\Lambda$, the points are  independent and uniformly distributed in $\Lambda$. 
Each point is marked by a mark distributed by $\Radius$ and all the marks are independent.    
\end{itemize}
We refer to \cite{daley_vere_jones} for details on Poisson point process.

For multi-index $\MultiIntensity=(\Intensity_1, \dots, \Intensity_\Types)$ and 
$\MultiRadius := (\Radius_1,\dots ,\Radius_\Types)$, 
$\MultiPoisson{\MultiIntensity,\MultiRadius} := \Poisson{\Intensity_1,\Radius_1}\otimes \dots \Poisson{\Intensity_\Types,\Radius_\Types}$ is the distribution of a multi-index Poisson point process.
For $\Lambda \subseteq \R^d$ bounded, we denote by $\Poisson{\Intensity,\Radius}_\Lambda$ (respectively $\MultiPoisson{\MultiIntensity,\MultiRadius}_\Lambda$) the restriction of  $\Poisson{\Intensity,\Radius}$(respectively $\MultiPoisson{\MultiIntensity,\MultiRadius}$) on $\Lambda \times \R^+$ (respectively $(\Lambda \times \R^+)^\Types$). 
Note that measures $\Poisson{\Intensity, \Radius}$ and 
$\MultiPoisson{\MultiParamPoisson}$ are \emph{stationary}, which means they are invariant under all translations of vector $x \in \R^\Dim$.

The connectivity properties of the Poisson point process play a crucial role in this study.
It changes drastically depending on an integrability condition, formalised in the following definition.
\begin{definition}\label{definition_int}
A family $\MultiRadius$ is said to satisfy the \emph{integrability assumption} if for every colour $i$,
\begin{align}
\label{eq_integrability_assumption}
\int_{\R^+ } r^\Dim \Radius_i(dr) <\infty.
\end{align}
If not satisfied we say that $\MultiRadius$ is \emph{not integrable} (or in the \emph{extreme case}). 
We say that  $\MultiRadius$ is \emph{completely non-integrable} if for every $1\le i \le q$,  $\int_{\R^+ } r^\Dim \Radius_i(dr) =\infty$.
\end{definition}
%This definition translates trivially to a measure $\Radius$ by considering it as a family of one element.

\subsection{Widom-Rowlinson measures}
The Widom-Rowlinson measures are defined with standard DLR equations requiring the probability measures to have prescribed conditional probabilities.
Let us first define the event $\A$ of \emph{authorized} (or \emph{allowed} ) configurations
\begin{align*}
\A = \{
\MultiConf \in \MultiConfSpace, \ \forall 1 \leq i<j \leq \Types, \
L(\Conf^i) \cap L(\Conf^j)= \emptyset 
\} \,.
\end{align*}
The Widom-Rowlinson specification on a bounded $\Lambda \subseteq \R^d$ with boundary condition $\MultiConf_{\Lambda^c}$ is
\begin{align*}
\Specification{\Lambda, \MultiConf_{\Lambda^c}}(d \MultiConf'_\Lambda)
: =
\frac{\1_{\A}(\MultiConf'_{\Lambda} \cup \MultiConf_{\Lambda^c})}{\PartitionFunction{\Lambda,\MultiConf_{\Lambda^c}}}
\MultiPoisson{\MultiIntensity, \MultiRadius}_{\Lambda} (d \MultiConf'_{\Lambda}),
\end{align*}
with
\begin{align*}
\PartitionFunction{\Lambda,\MultiConf_{\Lambda^c}} := 
\int_{\MultiConfSpace} \1_{\A}(\MultiConf'_{\Lambda} \cup \MultiConf_{\Lambda^c})
\MultiPoisson{\MultiParamPoisson}_{\Lambda} (d \MultiConf'_{\Lambda}).
\end{align*}

\begin{definition}
\label{def_WR}
A probability measure $\MultiP$ on $\MultiConfSpace$ is a Widom-Rowlinson measure of parameters $\MultiIntensity$ and $\MultiRadius$, written 
$\MultiP \in WR(\MultiParamPoisson)$, if $\MultiP$ is stationary and if for every bounded Borel set $\Lambda \subseteq \R^d$ and  every bounded measurable function $f$,
\begin{subequations}
\label{eq_def_wr}
\begin{align}
\label{eq_partition_function_non_deg_def_wr}
\PartitionFunction{\Lambda,\MultiConf_{\Lambda^c}} >0 \quad \MultiP(d \MultiConf)-a.s;
\end{align}
\begin{align}
\label{eq_dlr_definition_wr}
\int_{\MultiConfSpace} f \ d \MultiP 
=
\int_{\MultiConfSpace} \int_{\MultiConfSpace}
f(\MultiConf'_{\Lambda} \cup \MultiConf_{\Lambda^c} )
\Specification{\Lambda, \MultiConf_{\Lambda^c}}(d \MultiConf'_\Lambda)
\MultiP (d \MultiConf ).
\end{align}
\end{subequations}
\end{definition}
For every $\Lambda$ the equations \eqref{eq_dlr_definition_wr} are called DLR equations, named after Dobrushin, Lanford and Ruelle.
Thanks to \cite{georgii_livre}, a Widom-Rowlinson measure is a mixture of ergodic Widom-Rowlinson measures.

\subsection{Stochastic domination}
Let us discuss  stochastic domination, which is going to be a key element of several proofs in the paper.
Recall that an event $E \in \SigmaAlgebra$ is said \emph{increasing} if for $\Conf' \in E$ and $\Conf \supseteq \Conf'$, we have $\Conf \in E$.
Finally if $P$ and $P'$ are two probability measures on $\ConfSpace$, the measure $P$ is said to \emph{stochastically dominate} the measure $P'$, written $P' \preceq P$, if $P'(E) \leq P(E)$ for every increasing event $E \in \SigmaAlgebra$.
Those definitions naturally extend to the case of the multi-index configuration space $\MultiConfSpace$.

The following proposition is a direct application of 
\cite[Theorem 1.1]{georgii_kuneth} and gives a comparison between a Widom-Rowlinson measure and a Poisson point process.

\begin{proposition}[Stochastic domination]
\label{propo_stochastic_dom_poisson_wr}
For every bounded $\Lambda \subseteq \R^d$ and every boundary condition 
$\MultiConf_{\Lambda^c}$, we have 
$$\Specification{\Lambda, \MultiConf_{\Lambda^c}}(d \MultiConf'_\Lambda)
\preceq \MultiPoisson{\MultiParamPoisson}_{\Lambda}.$$

Furthermore for every $\MultiP \in WR( \MultiParamPoisson)$
$$
\MultiP  \preceq \MultiPoisson{\MultiParamPoisson}.
$$
\end{proposition}

\section{Results} \label{section_results}
This section states the main results of the present article.
\subsection{Existence}
The first question of interest in statistical physics where the objects are defined through prescribed conditional equations, namely the DLR equations \eqref{eq_dlr_definition_wr}, is the question of the existence of a probability measure solving those equations.
The following theorem gives a positive answer to this question.
\begin{theorem}
\label{theo_existence_wr}
For any $d\ge 1$ and  any parameters $\MultiIntensity$ and $\MultiRadius$, 
the set $WR(\MultiParamPoisson)$ is not empty.
\end{theorem}
The proof of this theorem follows a standard scheme used for several models 
\cite{dereudre_2009,dereudre_drouilhet_georgii,dereudre_houdebert}.
First, using the specific entropy, we build an accumulation point of a sequence of finite
volume  Gibbs  measures. 
The difficulty is then to prove that the accumulation  point satisfies the DLR equations \eqref{eq_dlr_definition_wr}.
This is done using the stochastic domination result of Proposition \ref{propo_stochastic_dom_poisson_wr}.
A detailed proof is given in Section \ref{section_preuve_theo_existence}.

\begin{remark}
The existence of a Widom-Rowlinson measure was already known in several cases.
First, in the cases of bounded radii, the Widom-Rowlinson interaction is finite range and therefore the existence is a consequence of a general result of Ruelle \cite{ruelle_livre_1969}.
Second, in the symmetric case where 
$\Intensity_1=\dots = \Intensity_\Types$ and
$\Radius_1=\dots = \Radius_\Types$, a Widom-Rowlinson measure can be built from a \emph{Continuum Random Cluster Model} where to each connected component is assigned a colour uniformly over the $\Types$ choices.
This relation is known as the \emph{Fortuin-Kasteleyn} representation.
The existence of the Continuum Random Cluster Model with random radii has been recently proved in \cite{dereudre_houdebert}, yielding the existence of a Widom-Rowlinson measure in the symmetric case.
\end{remark}

\subsection{Phase transition in the integrable case}
Now that Thereom \ref{theo_existence_wr} proves the existence of at least one Widom-Rowlinson measure, the second question concerns the uniqueness, non-uniqueness and consequently the phase transition between both regimes. It is usual, for Gibbs point processes with different type of particles, to show the uniqueness for small activities and non-uniqueness for large activity. 
In the integrable case (see definition \ref{definition_int}), we recover both regimes.

First, the following proposition proves the uniqueness for small activities.
\begin{proposition}
\label{prop_unicité_disagreement_percolation}
Write $\Intensity_i= \Intensity \Proportion_i$ where 
$\MultiProportion := (\Proportion_1, \dots , \Proportion_\Types)$ is a discrete probability measure.
If $\MultiRadius$ satisfies the integrability assumption (Definition \ref{definition_int}), then there exists an  unique Widom-Rowlinson measure $P$ in $WR(\Intensity \MultiProportion,\MultiRadius)$ as soon as 
$\Intensity \leq 
\PercolationThreshold \left(\Dim,\sum_i \Proportion_i \Radius_i \right)$, where $\PercolationThreshold (\Dim, \Radius)>0$ is the percolation threshold of the Poisson Boolean model in $\R^\Dim$ of radius measure $\Radius$.
\end{proposition}

\begin{proof}
From  Proposition \ref{propo_stochastic_dom_poisson_wr} we have the stochastic domination
$\Specification{\Lambda, \MultiConf_{\Lambda^c}}(d \MultiConf'_\Lambda)
\preceq \MultiPoisson{\MultiParamPoisson}_{\Lambda}$.
Therefore as a direct consequence of Theorem 3.2 in \cite{hofer-temmel_houdebert_2017} we have uniqueness of the Widom-Rowlinson measure as soon as the "single-type" Poisson Boolean model $\Poisson{\Intensity, \sum_i \Proportion_i \Radius_i}$ does not percolate. 
Thus the result. 
Let us note that $\PercolationThreshold (\Dim, \Radius)$ is positive for every $\Radius$ satisfying the integrability assumption $\int_{\R^+ } r^\Dim \Radius(dr) <\infty$ \cite{gouere_2008}. 
\end{proof}
%\begin{remark}
%In several case we can obtain a bound uniform with respect to $\MultiProportion$.
%For instance if $\Radius_1=\dots=\Radius_\Types$ then the percolation threshold
%$\PercolationThreshold \left(\Dim,\sum_i \Proportion_i \Radius_i \right)$ does not depend on $\MultiProportion$.
%Furthermore if all radius distribution $\Radius_i$ are stochastically dominated by some probability measure $\bar{\Radius}$ satisfying the integrability assumption, then we have uniqueness when 
%$z<\PercolationThreshold (\Dim, \bar{\Radius})$.
%\end{remark}
%%%%%%%%%%%%%%%%%%%%%%%%%%%%%%%%%%%%%%%%%%
%%%%%%%%%%%%%%%%%%%%%%%%%%%%%%%%%%%%%%%%%%
In the symmetric case a non-uniqueness result was proved, initially in \cite{chayes_kotecky,georgii_haggstrom} in the case of deterministic radii.
A non-trivial generalisation to the case of unbounded radii was proved in \cite{houdebert_2017} and is stated in the following proposition. Let us just mention that the unbounded radii case requires a delicate study of the percolation properties of the continuum random cluster model which does not dominate a Poisson process anymore.

\begin{proposition}
\label{proposition_transition_phase_cas_inte_perco}
Let us consider the symmetric case $\Intensity := \Intensity_1 = \dots = \Intensity_\Types$ and $\Radius := \Radius_1= \dots = \Radius_\Types$ in dimension $d\ge 2$. If $\Radius$ satisfies the integrability assumption and $\Radius(\{ 0 \})=0$, then for activities $\Intensity$ large enough, there exist $\Types$ distinct ergodic Widom-Rowlinson measures.
\end{proposition}
This result is a consequence of the Fortuin-Kasteleyn representation and the percolation of the Continuum Random Cluster Model for large activities $\Intensity$. 
As usual, the $\Types$ distinct ergodic Widom-Rowlinson measures corresponds to the distinct Gibbs measures with boundary condition $1\le i\le \Types$. 

\begin{remark}${}$
The assumption $\Radius (\{ 0 \})=0$ is an artefact of the proof of the percolation of the Continuum Random Cluster Model developed in \cite{houdebert_2017}.
In this paper the author emphasizes that the proof would carry out the same for radius measures $\Radius$ having small atoms in $0$, and he conjectures that the results would be true with the maximal assumption $Q(\{ 0 \})<1$. We do not investigate this generalization here.
\end{remark}

In the case of deterministic radii $r_1, r_2, \dots, r_q$, it is proved in \cite{bricmont_kuroda_lebowitz_1984} that for large activities $\Intensity_1, \Intensity_2, \ldots, \Intensity_{\Types -1}$ there exists 
$\Intensity_\Types>0$ such that the set of Widom-Rowlinson measures $WR((\Intensity_1,\Intensity_2,\ldots,\Intensity_\Types), (\delta_{r_1},\delta_{r_2},\ldots,\delta_{r_\Types}))$ exhibits at least $\Types$ extremal phases. 
This result is based on an extension of the Pirogov-Sinai theory of phase transitions in general lattice spin systems to continuum systems. 
In the case of non-symmetric random radii we do not know if such result holds. However it is reasonable to believe that in the case of bounded random radii, the Pirogov-Sinai machinery is feasible and similar results could be obtained. 

%
%
%Let us cite also \cite{mazel_suhov_stuhl} where the Pirogov-Sinai theory is used to characterize the phase space of Widom-Rowlinson measures with  deterministic collection of radii. 
%{\color{red} Je ne comprend pas cette phrase, je pense qu'il y a une erreur.
%As far as we know, for the continuum Widom-Rowlinson model, it is only proved the existence of such ordered phase (i.e phase with boundary condition $1\le i\le \Types$).} 
%In the completely non-integrable setting, studied in the next section, we prove the existence of $(\Types +1)$-th disordered phase.

%ystems goes over into a phase-diagram of the WidomRowlinson
%model at large fugacities z=(z o ..... z r_ ~). There is in z-space a
%point where the system has r-pure phases, lines with r-1 phases, two
%dimensional surfaces with r-2 phases, etc. 
%
%
%ystems goes over into a phase-diagram of the WidomRowlinson
%model at large fugacities z=(z o ..... z r_ ~). There is in z-space a
%point where the system has r-pure phases, lines with r-1 phases, two
%dimensional surfaces with r-2 phases, etc. 

\subsection{Existence of $\mathbf{\Types +1}$ extremal phases in the non integrable setting}

The main results of the paper are presented in this section where we investigate the phase diagram in the non-integrable setting. A central notion here is the monochromaticity or polychromaticity of Widom-Rowlinson measures. This is defined as follows.
\begin{definition}
Let $Mono$ be the event of \emph{monochromatic} configurations 
$\MultiConf \in \MultiConfSpace$ such that $\Conf_i= \emptyset$ for all $1\le i\le \Types$ excepted one index. 
Let $Poly$ be the set of \emph{polychromatic} configurations, meaning that 
$Poly=Mono^c$.
A probability measure $\MultiP$ on $\MultiConfSpace$ is said \emph{monochromatic} (respectively \emph{polychromatic}) if 
$\MultiP (Mono)=1$ (respectively $\MultiP(Poly)=1$).
\end{definition}

Let us note that, in the case of monochromatic $\MultiP$, the index $i$ such that $\omega_i\neq \emptyset$ can be random. In the case of radii $\MultiRadius$ satisfying the integrability assumption \eqref{eq_integrability_assumption}, it is clear that every Widom-Rowlinson measure 
$\MultiP \in WR( \MultiParamPoisson)$ is polychromatic. 
Therefore, the question of monochromaticity is relevant only in this non-integrable setting. Moreover, we know that such monochromatic Widom-Rowlinson measures exist in the non-integrable setting as mentioned in the next proposition. The proof is obvious and is not detailed here.

%this question is trivially answered by the following proposition.
%\begin{proposition}
%\label{proposition_polychromaticite_cas_inte}
%If $\MultiRadius$ satisfies the integrability assumption, then every Widom-Rowlinson measure 
%$\MultiP \in WR( \MultiParamPoisson)$ is polychromatic.
%\end{proposition}
%\begin{proof}
%Without loss of generality let us assume that $\MultiP$ is ergodic.
%By Proposition \ref{propo_stochastic_dom_poisson_wr} we have $\MultiP \preceq \MultiPoisson{\MultiParamPoisson}$ and  it is known, using for example the events $\upsk$ and the Lemma \ref{lemme_upsilon_proba}, that for $\Lambda' \subseteq \Lambda$ bounded with $\Lambda'$ not too small and $\Lambda$ large enough, with positive probability we have $L(\MultiConf_{\Lambda^c}) \cap \Lambda' = \emptyset$.
%Furthermore using the DLR($\Lambda$) equation \eqref{eq_dlr_definition_wr} one can show that $\MultiP (Poly)>0$.
%This implies that $\MultiP(Poly)=1$ since the event $Poly$ is translation invariant.
%\end{proof}

\begin{proposition}\label{proposition_extreme_evident}
For every $1\le i \le \Types$, such that $\int_{\R^+} r^\Dim \Radius_i(dr)=+\infty$, the Poisson point process 
$\Poisson{\bar{\MultiIntensity}^i , \MultiRadius}$ with 
$\bar{\MultiIntensity}^i=(0 \dots, 0, \Intensity_i,0,\dots,0)$, is an extremal phase of $WR(\MultiParamPoisson)$.
\end{proposition}

In particular, if $\MultiRadius$ is completely non-integrable (i.e. $\int_{\R^+} r^ \Dim \Radius_i(dr)=+\infty$ for every  $1\le i \le \Types$)  then $WR(\MultiParamPoisson)$ has $\Types$ extremal monochromatic Widom-Rowlinson measures which correspond to the usual $\Types$ ordered phases. 
In the next theorem we prove that, if the activity is small enough, there always exists a polychromatic Widom-Rowlinson measure (without any integrability assumption). 
Therefore, the existence of a $\Types +1$-th extremal phase follows in the completely non-integrable setting.

\begin{theorem}
\label{theo_polychromaticite_cas_extreme_faible activite}
Write $\Intensity_i = \Intensity \Proportion_i$ with $\MultiProportion := (\Proportion_i)_i$ being a discrete probability distribution. Then, in any dimension $d\ge 1$
\begin{enumerate}
\item for every $\MultiProportion$ such that $\max_i \Proportion_i <1$, 
there exists $\Intensity_c^{\MultiProportion}$ such that for all 
$\Intensity<\Intensity_c^{\MultiProportion}$, the set $WR(\Intensity \MultiProportion , \MultiRadius )$ contains at least one polychromatic Widom-Rowlinson measure;
\item 
the constant $\Intensity_c^{\MultiProportion}$ can be chosen uniform in $\MultiProportion$ satisfying for some 
$0< \Proportion_{\max}<1$,
\begin{align}
\label{eq_condition_proportion}
\forall i \in \{1, \dots , \Types \}, \quad
 \Proportion_i \leq \Proportion_{\max}.
\end{align} 
\end{enumerate}
\end{theorem}
The sketch of the proof is as follows. We build first an accumulation point of the sequence of finite volume Widom-Rowlinson measures with free boundary condition. 
Then we show, for small activity, that the specific entropy of this measure is smaller than the specific entropy of every monochromatic stationary probability measure. Therefore this measure is not monochromatic and the theorem follows. The details of the proof are given in Section \ref{section_preuve_non_monochromaticite_petite activite}. The assumption $\max_i \Proportion_i <1$ (respectively $\Proportion_{\max}<1$) ensures that at least two of the $\Proportion_i$ are positive.
This is a natural assumption in order to have polychromaticity.

\begin{corollary} We assume that $\MultiProportion := (\Proportion_i)_i$ is a discrete probability with non-null coordinate and that $\MultiRadius$ is completely non integrable (i.e. for every  $1\le i \le q$,  $\int_{\R^+ } r^\Dim \Radius_i(dr) =\infty$). 
Then $WR(\Intensity \MultiProportion , \MultiRadius )$ has at least $\Types+1$ extremal phases for $\Intensity$ small enough.
\end{corollary}

The assumption "$\Intensity$ small" appears in the proof of Theorem \ref{theo_polychromaticite_cas_extreme_faible activite} as an artificial assumption needed to ensure that the accumulation point is not monochromatic.
However it is our belief that this assumption is crucial and that for large activities all Widom-Rowlinson measures would be monochromatic. 
This is formalized in the following conjecture.
\begin{conjecture}
\label{conjecture_monochromaticite_cas_extreme}
In the non-integrable case (i.e. there exists  $1\le i \le q$ such that  $\int_{\R^+ } r^\Dim \Radius_i(dr) =\infty$), for activities $\Intensity$ large enough, every Widom-Rowlinson measures is monochromatic.
\end{conjecture}

Note that if the conjecture is true, it would imply a phase transition result since the set $ WR( \MultiParamPoisson)$ would have exactly $\Types$ extremal Gibbs measures for $\MultiIntensity$ large and at  least  $\Types+1$ extremal Gibbs measures for $\MultiIntensity$ small. 
Our belief in this conjecture is based on a similar conjecture for the continuum random cluster presented in \cite{dereudre_houdebert} and for which a heuristic proof is given. 
Moreover simulations in the sense of the conjecture have been implemented in \cite[Section III.2]{houdebert_these}. 

We have not succeeded to prove the conjecture but our last result is a first step towards it, by proving a weaker version of the conjecture.

Indeed we show that the symmetric Widom-Rowlinson measure on $ \R^{\Dim}$ with free boundary condition and non-integrable radii is monochromatic if and only the activity $\Intensity$ is large enough. Note that the scale we use in the thermodynamic limit is not symmetric since one direction is favoured. 
Unfortunately  we are not able to extend the result for every thermodynamic limit, in particular, when the volume $\Lambda_n$ is simply an hypercube.

For every $k>0$ and $n\ge 1$, let $\Lambda_n^{(k)}:= ]-n,n] \times [0,k]^{\Dim -1}$ and $\Lambda_n^{(k+)}:= ]0,n] \times [0,k]^{\Dim -1}$. 
Let us fix a sequence $(k_n)_{n\ge 1}$ of positive integers such that $k_n\to +\infty$ and $(k_n)$ is negligible with respect to $(\log(n))_{n\ge 1}$. Now, for any $n\ge 1$,  we consider the Widom-Rowlinson measure on $\Lambda_n^{(k_n)}$ with  free boundary condition

$$\MultiP_n^{\text{free}}(d \MultiConf):= 
\frac{\1_\A(\MultiConf)}{\FinitePartitionFunction}\MultiPoisson{\MultiParamPoisson}_{\Lambda_n^{(k_n)}}(d \MultiConf).$$

As in the proof of Theorem \ref{theo_existence_wr}, we introduce its stationary version $\bar{\MultiP}_n^{\text{free}}$. First $\hat{\MultiP}_n^{\text{free}}:= 
\underset{x \in I_n}{\otimes} 
\MultiP_n^{\text{free}} \circ \tau_x^{-1}$ and finally 
$$\bar{\MultiP}_n^{\text{free}} := \frac{1}{\Leb (\Lambda_n^{(k_n)})}
\int_{\Lambda_n^{(k_n)}} \hat{\MultiP}_n^{\text{free}} \circ \tau_x^{-1} dx,$$
where $I_n:= 2n \Z \times ( k_n \Z)^{\Dim -1}$ and where $\tau_x$ is the translation operator of vector $x\in\R^\Dim$.

As in the proof of Theorem \ref{theo_existence_wr}, it is easy to show that the sequence 
$(\bar{\MultiP}_n^{\text{free}})$ admits at least one accumulation point $\MultiP^{\text{free}}$, with respect to the local convergence topology. The following theorem is our phase transition result involving these accumulation points.

\begin{theorem}
\label{theo_conjecture_preuve_faible}
We assume that we are in the symmetric case 
$\Intensity := \Intensity_1 = \dots = \Intensity_\Types$,
$\Radius := \Radius_1= \dots = \Radius_{\Types}$ in dimension $d\ge 1$,
with $\Radius$ satisfying the two following conditions
\begin{subequations}
\label{eq_conjecture_conditions}
\begin{align} 
\label{eq_conjecture_cond1}
\int_1^\infty \exp \left(
 - \int_1^u \Radius (]r,\infty[) dr
\right) du < \infty;
\end{align}
\begin{align}
\label{eq_conjecture_cond2}
\Radius (\{0\})<\frac{1}{\Types}.
\end{align}
\end{subequations}

Then, for $\Intensity$ large enough, $(\bar{\MultiP}_n^{\text{free}})$ converges (without passing by a subsequence) to  $\MultiP^{\text{free}}$ which is the mixture $\sum_{i=1}^q \frac{1}{q} \MultiPoisson{\bar{\MultiIntensity}^i ,\MultiRadius}$, with $\bar{\MultiIntensity}^i=(0 \dots, 0, \Intensity_i,0,\dots,0)$
 (i.e. $\MultiP^{\text{free}}$ is monochromatic with equal probability of having any color). In opposite, when $z$ is small enough, every accumulation point of  $(\bar{\MultiP}_n^{\text{free}})$ is not monochromatic and therefore it is not a mixture of the  monochromatic phase $\MultiPoisson{\bar{\MultiIntensity}^i ,\MultiRadius}$, $1\le i\le q$.
\end{theorem}

Let us note that the assumptions \eqref{eq_conjecture_cond1} and \eqref{eq_conjecture_cond2} are purely technical and probably not really necessary. Indeed, they are related to the renewal strategy we used to proof the first part of the Theorem. We believe that the assumption \eqref{eq_conjecture_cond1} could  be replaced by $\int_{\R^+ } r^\Dim  Q(dr) =\infty$ and the assumption \eqref{eq_conjecture_cond2} by $\Radius (\{0\})<1$.

The proof of the theorem is based on the Fortuin-Kasteleyn representation in order to transfer the problem from the Widom-Rowlinson setting to the Continuum Random Cluster setting.
We will prove that the sequence of finite volume Continuum Random Cluster measures converges towards the Poisson Boolean model which covers the whole space $\R^\Dim$, proving that $\MultiP^{\text{free}}$ is monochromatic with equal probability for each color by symmetry of the model.
This is done by bounding the mean number of connected components of a Continuum Random Cluster measure, using a fine renewal argument.
The detailed proof is given in Section \ref{section_preuve_conjecture_faible}.

\begin{remark}

In the case of dimension $\Dim=1$, the sequence $(k_n)_n$ plays no role and the probability measures
$\MultiP_n^{\text{free}}$,
$\hat{\MultiP}_n^{\text{free}}$ and 
$\bar{\MultiP}_n^{\text{free}}$ are just the measures 
$\MultiP_n$, $\hat{\MultiP}_n$ and  $\bar{\MultiP}_n$ 
which will be introduced in the proof of Thereom \ref{theo_existence_wr} 
and Theorem \ref{theo_polychromaticite_cas_extreme_faible activite}, see Section \ref{section_preuve_theo_existence} and Section \ref{section_preuve_non_monochromaticite_petite activite}. Let us note also that the Widom-Rowlinson model with non-integrable radii exhibits a phase transition in dimension one. It is quite unusual and due to the very long range of the interaction. 
\end{remark}

\section{Proof of Theorem \ref{theo_existence_wr}} 
\label{section_preuve_theo_existence}
First let us consider the extreme case where at least one radius measure, let say $\Radius_i$, satisfy 
$\int_{\R^+} r^\Dim \Radius_i(dr)=\infty$.
It is known, see \cite{stoyan_kendall_mecke} for instance, that the Poisson Boolean model covers almost surely the all space $\R^\Dim$.
In particular for every bounded $\Lambda \subseteq \R^\Dim$ and $\Poisson{\Intensity_i,\Radius_i}$-almost every configuration $\Conf$, we have 
$L(\Conf_{\Lambda^c})=\R^\Dim$.
Let us consider the probability measure on $\MultiConfSpace$ with one marginal being $\Poisson{\Intensity_i,\Radius_i}$ and the others producing almost surely empty configurations.
This is a Poisson Point process $\Poisson{\bar{\MultiIntensity}^i , \MultiRadius}$
 with only one non-zero intensity $\Intensity_i$, i.e.  $\bar{\MultiIntensity}^i:= (0, \dots ,0, \Intensity_i, 0, \dots ,0)$.
This probability measure trivially satisfies conditions \eqref{eq_def_wr} of Definition \ref{def_WR} and is therefore a Widom-Rowlinson measure of parameters 
 $\MultiIntensity$ and $\MultiRadius$.

So from now on we consider the case of $\MultiRadius$ satisfying the integrability assumption \eqref{eq_integrability_assumption}.

\subsection{Construction of a good cluster point}
To build a Widom-Rowlinson measure, consider  a sequence of Widom-Rowlinson measures on  the bounded boxes $\Lambda_n$ with free boundary condition defined as
\begin{align}
\MultiP_n(d\MultiConf) 
=
\Specification{\Lambda_n, \emptyset}(d\MultiConf) 
= \frac{\1_{\A}(\MultiConf) }{\FinitePartitionFunction{}} 
\MultiPoisson{\MultiParamPoisson}_{\Lambda_n}(d\MultiConf_{\Lambda_n}),
\nonumber
\end{align}
with $\Lambda_n :=]-n,n]^\Dim$ and 
$\FinitePartitionFunction{}= \int_{\MultiConfSpace} \1_{\A}(\MultiConf)
\MultiPoisson{\MultiParamPoisson}_{\Lambda_n}(d\MultiConf_{\Lambda_n})$.
Then consider 
$\hat{\MultiP}_n = \underset{i \in 2n \Z^d}{\otimes} \MultiP_n \circ \tau_{i}^{-1}$ 
and 
$\bar{\MultiP}_n=\frac{1}{\Leb(\Lambda_n)}\int_{\Lambda_n} 
( \hat{\MultiP}_n \circ \tau_x^{-1} ) dx$,
where $\tau_x$ is the translation operator of vector $x\in \R^\Dim$.
The measures $\bar{\MultiP}_n$ are by construction stationary.
The aim is to find an accumulation point for the sequence $(\MultiP_n)$ with respect to the local convergence topology defined in the next definition.
\begin{definition}
A measurable function $f: \MultiConfSpace \to \R$ is said to be \emph{local} if there exists a bounded $\Lambda \subseteq \R^\Dim$ such that $f(\MultiConf)=f(\MultiConf_{\Lambda})$ for all $\MultiConf \in \MultiConfSpace$.

A sequence $(\tilde{\MultiP}_n)$ converges, with respect to the local convergence topology, towards $\tilde{\MultiP}$, if for every bounded local function $f$ we have
\begin{align*}
\int_{\MultiConfSpace} f(\MultiConf) \tilde{\MultiP}_n(d\MultiConf) 
\underset{n \to \infty}{\longrightarrow}
\int_{\MultiConfSpace} f(\MultiConf) \tilde{\MultiP}(d\MultiConf).
\end{align*}
\end{definition}
A very convenient tool for proving the existence of an accumulation point is the specific entropy.
It was introduced in \cite{georgii_livre} and is defined in the following definition.
\begin{definition}
For a stationary probability measure $\MultiP$ we define the specific entropy of $\MultiP$, written $\Entropy (\MultiP)$, as the following limit:
\begin{align}
\label{eq_specific_entropy_def}
\Entropy (\MultiP)
=
\lim_{n \to \infty} \
\frac{1}{\Leb (\Lambda_n)} 
\Entropy_{\Lambda_n}(\MultiP | \MultiPoisson{\MultiParamPoisson}),
\end{align}
with $\Entropy_{\Lambda_n}(\MultiP | \MultiPoisson{\MultiParamPoisson})$ being the relative entropy of $\MultiP$, with respect to $\MultiPoisson{\MultiParamPoisson}$, defined as
\begin{align}
\Entropy_{\Lambda_n}(\MultiP | \MultiPoisson{\MultiParamPoisson})=
\left\lbrace
\begin{array}{ccc}
\int_{\MultiConfSpace} g \ \log (g) \ d\MultiPoisson{\MultiParamPoisson}_{\Lambda_n}  & \mbox{if}& 
\MultiP_{\Lambda_n} \ll \MultiPoisson{\MultiParamPoisson}_{\Lambda_n}, 
\ g=\frac{d\MultiP_{\Lambda_n}}{d\MultiPoisson{\MultiParamPoisson}_{\Lambda_n}}
\\
+ \infty & \mbox{else} &
\end{array}\right. .
\nonumber
\end{align}
\end{definition}
The stationarity ensures the  convergence of the limit in \eqref{eq_specific_entropy_def}.
This is the reason why we introduce a stationary version $\bar{\MultiP}_n$ of the finite volume Widom-Rowlinson measure $\MultiP_n$. The following proposition ensures the compactness of the level sets of the specific entropy.
\begin{proposition}[Proposition 2.6 \cite{georgii_zessin}] 
\label{propo_compactness_entropy_specifique}
With respect to the local convergence topology induced on the stationary probability measures on $\MultiConfSpace$, we have
\begin{enumerate}
\item $\Entropy$ is affine;
\item $\Entropy$ is lower semi-continuous;
\item the set $ \{ \MultiP$ stationary, $ \Entropy (\MultiP) \leq C \}$ is compact for every positive real number $C$.
\end{enumerate}
\end{proposition}

\begin{proposition}
\label{propo_borne_entropi_spécifique_wr}
For all $n$ we have
\begin{align*}
\Entropy (\bar{\MultiP}_n)
\leq \Intensity_1 + \dots + \Intensity_\Types.
\end{align*}
\end{proposition}
\begin{proof}
First using the fact that the specific entropy is affine, see Proposition \ref{propo_compactness_entropy_specifique}, we have
\begin{align}
\Entropy(\bar{\MultiP}_n)
= \frac{1}{\Leb (\Lambda_n)}
\Entropy_{\Lambda_n}(\MultiP_n|\MultiPoisson{\MultiParamPoisson}).
\nonumber
\end{align}
Now using the definition of the specific entropy and standard bounds on the partition function, we get
\begin{align}
\Entropy_{\Lambda_n}(\MultiP_n|\MultiPoisson{\MultiParamPoisson})
= -\log (\FinitePartitionFunction{})
\leq (\Intensity_1 + \dots + \Intensity_\Types) \Leb (\Lambda_n),
\nonumber
\end{align}
which leads to the expected result.
\end{proof}
Using Proposition \ref{propo_compactness_entropy_specifique} and Proposition \ref{propo_borne_entropi_spécifique_wr}, we obtain the existence of an accumulation point $\MultiP$ of the sequence $(\bar{\MultiP}_n)$.
For the rest of the proof we are, for convenience of notation, omitting to take a subsequence when taking $n$ go to infinity.
We now have a good candidate for being a Widom-Rowlinson measure.

Then we have to prove equations \eqref{eq_def_wr}.
This is done for the symmetric case in the PhD manuscript \cite{houdebert_these} and we are here adapting the proof to the non-symmetric case.

In the next proposition we prove that the measure $\MultiP$ produces almost surely authorized configuration.
This trivially implies that condition \eqref{eq_partition_function_non_deg_def_wr} is fulfilled.
\begin{proposition}
\label{propo_candidat_non_degenere}
We have $\MultiP(\A)=1$, and therefore
\begin{align}
\PartitionFunction{\Lambda,\MultiConf_{\Lambda^c}} \geq 
\exp(-\Intensity \Leb(\Lambda) ) 
\ \MultiP(d \MultiConf) - \text{almost surely}.
\nonumber
\end{align}
\end{proposition}
\begin{proof}
The event $\A$ is not local and we cannot use directly the local convergence.
But this event could be called "almost local" since for every configuration $\MultiConf$,
\begin{align}
\1_{\A}(\MultiConf_{\Lambda_k})
\underset{k \to \infty}{\longrightarrow}
\1_{\A}(\MultiConf).
\nonumber
\end{align}
Therefore we have
\begin{align}
\MultiP(\A) &= 
\int_{\MultiConfSpace} \1_{\A}(\MultiConf) \MultiP(d \MultiConf)
= \nonumber
\underset{k \to \infty}{\lim}\int_{\MultiConfSpace} 
\1_{\A}(\MultiConf_{\Lambda_k}) \MultiP(d \MultiConf)
\\ & =
\underset{k \to \infty}{\lim} \underset{n \to \infty}{\lim}
\int_{\MultiConfSpace} 
\1_{\A}(\MultiConf_{\Lambda_k}) \bar{\MultiP}_n(d \MultiConf)
\nonumber
\end{align}
with
\begin{align}
\int_{\MultiConfSpace} 
\1_{\A}(\MultiConf_{\Lambda_k}) \bar{\MultiP}_n(d \MultiConf)
& =
\frac{1}{\Leb (\Lambda_n)} \int_{\Lambda_n}
\int_{\MultiConfSpace} 
\1_{\A}(\MultiConf_{\Lambda_k}) \hat{\MultiP}_n \circ \tau_x ^{-1}(d \MultiConf) dx  
\nonumber \\ & =
\frac{1}{\Leb (\Lambda_n)} \int_{\Lambda_n}
\int_{\MultiConfSpace} 
\1_{\A} (\MultiConf_{\tau_x (\Lambda_k)}) \hat{\MultiP}_n (d \MultiConf) dx .
\nonumber
\end{align}
For $n >k$, we have $\tau_x(\Lambda_k) \subseteq \Lambda _n$ as soon as $x \in [k-n,n-k]^d$ and so
\begin{align}
\frac{1}{\Leb (\Lambda_n)}
\int_{\Lambda_n} \int_{\MultiConfSpace} 
\1_{\A}(\MultiConf_{\tau_x (\Lambda_k)}) \hat{\MultiP}_n(d \MultiConf)dx
 \geq
 \frac{\Leb ([k-n,n-k]^d)}{\Leb (\Lambda_n)},
\nonumber
\end{align}
which tends to $1$ as $n$ goes to infinity.
The result is proved.
\end{proof}

\subsection{The cluster point statisfies \eqref{eq_dlr_definition_wr}}
Starting now we fix $\Lambda \subseteq \R^\Dim$.
We are going to prove that $\MultiP$ satisfies the DLR($\Lambda$) equation.
To do so we need to modify the sequence $(\bar{\MultiP}_n)$, defining a new sequence $(\tilde{\MultiP}_n^{\Lambda})$.
This new sequence will be asymptotically equivalent to the former one and each 
$\tilde{\MultiP}_n^{\Lambda}$ will satisfy the DLR($\Lambda$) equation \eqref{eq_dlr_definition_wr}.
Finally by considering good "localizing" events, we will be able to pass the DLR property through the limit.

Consider
\begin{align}
\tilde{\MultiP}_n^{\Lambda}=
\frac{1}{\Leb(\Lambda_n)}\int_{\Lambda_n} 
\1_{\Lambda \subseteq \tau_x(\Lambda_n)} \times ( \MultiP_n \circ \tau_x^{-1} ) dx.
\nonumber
\end{align}
The measures $\tilde{\MultiP}_n^{\Lambda}$ are not probability measures but satisfy good properties, see the following proposition.
\begin{proposition}
\label{propo_conv_local_et_dlr}
For each bounded local function $f : \MultiConfSpace \to \R$ we have
\begin{align}
\left|  \int_{\MultiConfSpace} f \ d\tilde{\MultiP}_n^{\Lambda} - 
\int_{\MultiConfSpace} f \ d\MultiP_n \right| \to 0
\nonumber
\end{align}
as $n\mapsto +\infty$, which implies that $\MultiP$ is an accumulation point of the sequence
 ($\tilde{\MultiP}_n^{\Lambda}$).
  
Furthermore the measure  $\tP_n^{\Lambda}$ satisfies the DLR($\Lambda$) equation \eqref{eq_dlr_definition_wr}.
\end{proposition}
We are omitting the proof of this standard result.
The first point is done in \cite{dereudre_2009} for the \emph{Quermass-interaction model}, and the second point is a consequence of the compatibility of the Gibbs
 specification $\Specification{\Lambda, \MultiConf_{\Lambda^c}}$.
 And both points are done in the PhD manuscript \cite{houdebert_these}.

Now let fix a measurable function $f$ bounded by $1$.
By the structure of $\MultiSigmaAlgebra$ we can consider $f$ to be local.
We are going to prove that for each $\epsilon>0$, the quantity
\begin{align}
\delta = \left| \int_{\MultiConfSpace} \ f d \MultiP  
- \int_{\MultiConfSpace} \int_{\MultiConfSpace} 
f(\MultiConf'_{\Lambda} \cup \MultiConf_{\Lambda^c})
\frac{\1_{\A}(\MultiConf'_{\Lambda} \cup \MultiConf_{\Lambda^c})}
{\PartitionFunction{\Lambda, \MultiConf_{\Lambda^c}}} 
\MultiPoisson{\MultiParamPoisson}_{\Lambda}(d\MultiConf'_{\Lambda})
 \MultiP (d\MultiConf) \right|
\nonumber
\end{align}
is bounded by $7\epsilon$. 

The function
$f_{\Lambda}(\MultiConf) :=
 \int_{\MultiConfSpace} 
f(\MultiConf'_{\Lambda} \cup \MultiConf_{\Lambda^c})
\frac{\1_{\A}(\MultiConf'_{\Lambda} \cup \MultiConf_{\Lambda^c})}
{\PartitionFunction{\Lambda, \MultiConf_{\Lambda^c}}} 
\MultiPoisson{\MultiParamPoisson}_{\Lambda}(d\MultiConf'_{\Lambda})$ is in general not local, which is the main obstacle in proving the result.
We cannot directly use the local convergence, and the piecewise convergence used in the proof of Proposition \ref{propo_candidat_non_degenere} is not good enough as we need a more uniform convergence.

Let us consider the event
\begin{align}
U_{r_1}=\{\MultiConf, \forall (x,r) \in \MultiConf_{\Lambda}, r \leq r_1 \}.
\nonumber
\end{align}
\begin{lemma}
\label{lemme_preuve_dlr_WR_introduction_petite_boule}
For $r_1$ large enough and for all $\MultiConf_{\Lambda^c} \in \A$ we have
\begin{align}
\left|
\int_{\MultiConfSpace} f(. \cup \MultiConf_{\Lambda^c})
\frac{\1_{\A}(. \cup \MultiConf_{\Lambda^c})}
{\PartitionFunction{\Lambda, \MultiConf_{\Lambda^c}} } 
d\MultiPoisson{\MultiParamPoisson}_{\Lambda}
-
\int_{\MultiConfSpace} f(. \cup \MultiConf_{\Lambda^c})
 \1_{U_{r_1}}(.) 
\frac{\1_{\A}(. \cup \MultiConf_{\Lambda^c})}
{\zr} 
d \MultiPoisson{\MultiParamPoisson}_{\Lambda}
\right| \leq \epsilon,
\nonumber
\end{align}
where $\zr : = \int_{\MultiConfSpace} \1_{U_{r_1}}
(\MultiConf'_{\Lambda}) 
\1_{\A}(\MultiConf'_{\Lambda} \cup \MultiConf_{\Lambda^c})
\MultiPoisson{\MultiParamPoisson}_{\Lambda}(d\MultiConf'_{\Lambda})$
is the \emph{modified partition function}.

The constant $r_1$ can be chosen to also have 
$\MultiPoisson{\MultiParamPoisson}_{\Lambda}(U_{r_1}^c) \leq \epsilon$.
\end{lemma}
The proof of Lemma \ref{lemme_preuve_dlr_WR_introduction_petite_boule} is done in Section \ref{section_preuves_lemmes_techniques}.
Using Lemma \ref{lemme_preuve_dlr_WR_introduction_petite_boule} we get
\begin{align}
\delta \leq \epsilon + \left| \int_{\MultiConfSpace} 
f \ d \MultiP  - \int _{\MultiConfSpace} \int_{\MultiConfSpace} f(\MultiConf'_{\Lambda} \cup \MultiConf_{\Lambda^c})
 \1_{U_{r_1}}(\MultiConf'_{\Lambda})
\frac{\1_{\A}(\MultiConf'_{\Lambda} \cup \MultiConf_{\Lambda^c})}{\zr}
 \MultiPoisson{\MultiParamPoisson}_{\Lambda}(d\MultiConf'_{\Lambda}) 
 \MultiP (d\MultiConf) \right|.
 \nonumber
\end{align}
Now in order to "localize" (with respect to $\MultiConf$) the functions $\1_{\A}$ and $\zr$ we introduce the events  
\begin{align}
\Upsilon_k = \{
\MultiConf \in \MultiConfSpace, \ \forall (x,r) \in \MultiConf,
\ r \leq  \frac{|x|}{2} + k
\}.
\nonumber
\end{align}

Let us note that a ball in a configuration in $\Upsilon_k$ has a radius smaller than the half of the distance of the centre from the origin (up to an additive fix constant $k$). Then, when the centre is far from the origin, the full ball is far from the origin as well. It is the reason why  $\Upsilon_k$ localizes the interaction. Now the next lemma claims that   $\Upsilon_k$ has a high probability when $k$ is large enough.

\begin{lemma}
\label{lemme_preuve_dlr_ups_total}
For $k$ large enough we have:
\begin{enumerate}
\item
$\MultiP(\Upsilon_k^c) \leq \epsilon$;
\item $\tilde{\MultiP}_n^{\Lambda}(\Upsilon_k^c) \leq \epsilon$  for each $n$;
\item $\MultiPoisson{\MultiParamPoisson}(\Upsilon_k^c) \leq \epsilon$.
\end{enumerate}
\end{lemma}
The proof of Lemma \ref{lemme_preuve_dlr_ups_total} is done in Section \ref{section_preuves_lemmes_techniques}.
Using this lemma we have
\begin{align*}
%\small
\delta \leq 2 \epsilon +&
\bigg| \int_{\MultiConfSpace} f \ d \MultiP  
\nonumber \\
& - 
\int_{\MultiConfSpace}  \int_{\MultiConfSpace} 
 f(\MultiConf'_{\Lambda} \cup \MultiConf_{\Lambda^c})
\1_{U_{r_1} }(\MultiConf'_{\Lambda})
\1_{\Upsilon_k}(\MultiConf)
\frac{\1_{\A}(\MultiConf'_{\Lambda} \cup \MultiConf_{\Lambda^c})}{\zr}
\MultiPoisson{\MultiParamPoisson}_{\Lambda}(d\MultiConf'_{\Lambda}) 
\MultiP (d\MultiConf)
 \bigg|.
\end{align*}
Now with the introduction of the events $U_{r_1}$ and $\Upsilon_k$, the next lemma enables us to "localize" the integrated functions.
\begin{lemma}
\label{lemme_preuve_dlr_localisation}
There exists $\Delta$ bounded, depending on $r_1$ and $k$, 
such that for every $\MultiConf'_{\Lambda} \in U_{R_1}$ and 
$\MultiConf \in \Upsilon_k \cap \A$,
\begin{align}
\1_{\A}(\MultiConf'_{\Lambda} \cup \MultiConf_{\Lambda^c})
=
\1_{\A}(\MultiConf'_{\Lambda} \cup \MultiConf_{\Delta \setminus \Lambda})
\text{ and }
\zr =
\boldsymbol{Z}_{ R_1}
(\Lambda, \MultiConf_{\Delta \setminus \Lambda}).
\nonumber
\end{align}
\end{lemma}
The proof of Lemma \ref{lemme_preuve_dlr_localisation} is done is Section \ref{section_preuves_lemmes_techniques}.
With this Lemma we have
\begin{align}
\small
\delta &\leq  2 \epsilon +
\bigg| \int_{\MultiConfSpace} f d\MultiP  
\nonumber \\
&\quad \ -
\int_{\MultiConfSpace}  \int_{\MultiConfSpace} 
 f(\MultiConf'_{\Lambda} \cup \MultiConf_{\Lambda^c})
\1_{U_{r_1} }(\MultiConf'_{\Lambda})
\1_{\Upsilon_k}(\MultiConf)
\frac{\1_{\A}(\MultiConf'_{\Lambda} \cup 
\MultiConf_{\Delta \setminus  \Lambda})}
{\boldsymbol{Z}_{r_1}(\Lambda, \MultiConf_{\Delta \setminus \Lambda})}
\MultiPoisson{\MultiParamPoisson}_{\Lambda}(d\MultiConf'_{\Lambda}) 
\MultiP (d\MultiConf)
 \bigg| \nonumber
 \\ & \leq 
4 \epsilon +
\bigg| \int_{\MultiConfSpace} f d \tilde{\MultiP}_n^{\Lambda} 
\nonumber \\
&\quad \quad  \quad - 
\int_{\MultiConfSpace}  \int_{\MultiConfSpace} 
 f(\MultiConf'_{\Lambda} \cup \MultiConf_{\Lambda^c})
\1_{U_{r_1}}(\MultiConf'_{\Lambda})
\frac{\1_{\A}(\MultiConf'_{\Lambda} \cup \MultiConf_{\Delta \setminus  \Lambda})}
{\boldsymbol{Z}_{r_1}(\MultiConf_{\Delta \setminus \Lambda})}
\MultiPoisson{\MultiParamPoisson}_{\Lambda}(d\MultiConf'_{\Lambda}) 
\tilde{\MultiP}_n^{\Lambda} (d\MultiConf)
 \bigg|,
\nonumber
\end{align}
where the last inequality comes from Lemma \ref{lemme_preuve_dlr_ups_total} and the local convergence of Proposition \ref{propo_conv_local_et_dlr}, for $n$ large enough fixed from now on.
Now using again Lemma  \ref{lemme_preuve_dlr_WR_introduction_petite_boule}  and Lemma \ref{lemme_preuve_dlr_ups_total},
we obtain
\begin{align}
\delta &\leq  7 \epsilon +
\underbrace{\left| \int_{\MultiConfSpace} f d \tilde{\MultiP}_n^{\Lambda}  - 
\int_{\MultiConfSpace}  \int_{\MultiConfSpace}  
%\1_{U_{r_1}}(\MultiConf'_{\Lambda})
f(\MultiConf'_{\Lambda} \cup \MultiConf_{\Lambda^c})
\frac{\1_{\A}(\MultiConf'_{\Lambda} \cup \MultiConf_{\Lambda^c})}
{\PartitionFunction{\Lambda, \MultiConf_{\Lambda^c}}}
\MultiPoisson{\MultiParamPoisson}_{\Lambda}(d\MultiConf'_{\Lambda}) 
\tilde{\MultiP}_n^{\Lambda} (d\MultiConf)
 \right|}_{=0 \text{ thanks to Proposition\ref{propo_conv_local_et_dlr}}}.
\nonumber
%\\ & =  6 \epsilon +
%\left| \int_{\MultiConfSpace} f d \tilde{\MultiP}_n^{\Lambda}  - 
%\int_{\MultiConfSpace}   \1_{U_{r_1}}(\MultiConf_{\Lambda})
%f(\MultiConf) \tilde{\MultiP}_n^{\Lambda} (d\MultiConf)
% \right| \nonumber
%\\ &\leq  6 \epsilon +
%\tilde{\MultiP}_n^{\Lambda} (U_{r_1}^c) 
%\leq 7 \epsilon,
%\nonumber
\end{align}

\subsection{Proof of the lemmas}
\label{section_preuves_lemmes_techniques}
\subsubsection{Proof of Lemma \ref{lemme_preuve_dlr_WR_introduction_petite_boule}}
First from standard computation we have
\begin{align}
\label{eq_preuve_dlr_poisson_petite_boule}
\Poisson{\MultiParamPoisson} (U_{r_1}^c)
=1- \exp \left(-\Leb (\Lambda )
\sum_i \Intensity_i \Radius_i([r_1,\infty[) \right)
\end{align}
which indeed can be as small as needed when $r_1$ is large enough.
\begin{align}
\label{eq_preuve_dlr_petite_boule_dans_dlr}
& \left|
\int_{\MultiConfSpace} f(. \cup \MultiConf_{\Lambda^c})
\frac{\1_{\A}(. \cup \MultiConf_{\Lambda^c})}
{\PartitionFunction{\Lambda, \MultiConf_{\Lambda^c}} } 
d\MultiPoisson{\MultiParamPoisson}_{\Lambda}
-
\int_{\MultiConfSpace} f(. \cup \MultiConf_{\Lambda^c})
 \1_{U_{r_1}}(.) 
\frac{\1_{\A}(. \cup \MultiConf_{\Lambda^c})}
{\zr} 
d \MultiPoisson{\MultiParamPoisson}_{\Lambda}
\right| \nonumber
\\ & \hspace{0.2cm} \leq
\Specification{\Lambda, \MultiConf_{\Lambda^c}}(U_{r_1}^c)
+
\frac{\PartitionFunction{\Lambda, \MultiConf_{\Lambda^c}}
- \zr }
{\PartitionFunction{\Lambda, \MultiConf_{\Lambda^c}}}
 \nonumber
 \\ & \hspace{0.2cm} \leq
 \Poisson{\MultiParamPoisson} (U_{r_1}^c)
 + \exp(\Leb (\Lambda) \sum_i \Intensity_i) \Poisson{\MultiParamPoisson} (U_{r_1}^c),
\end{align}
where the last inequality comes from Proposition \ref{propo_stochastic_dom_poisson_wr} and Proposition \ref{propo_candidat_non_degenere}.
So by choosing $r_1$ large enough both quantities \eqref{eq_preuve_dlr_poisson_petite_boule} and \eqref{eq_preuve_dlr_petite_boule_dans_dlr} are smaller than $\epsilon$.

\subsubsection{Proof of Lemma \ref{lemme_preuve_dlr_ups_total}}
The events $\Upsilon_k^c$ are increasing.
Therefore from Proposition \ref{propo_stochastic_dom_poisson_wr} we have 
$\MultiP (\Upsilon_k^c) \leq \MultiPoisson{\MultiParamPoisson} (\Upsilon_k^c)$.
Furthermore we have
\begin{align*}
\tilde{\MultiP}_n^{\Lambda}(\Upsilon_k^c)
\ \leq  \ \MultiP_n(\Upsilon_k^c)
\ \leq \  \Poisson{\MultiParamPoisson}_{\Lambda}(\Upsilon_k^c)
\ \leq \ \Poisson{\MultiParamPoisson}(\Upsilon_k^c),
\end{align*}
where the third inequality is a consequence of Proposition \ref{propo_stochastic_dom_poisson_wr} applied to 
$\MultiP_n = \Specification{\Lambda_n, \emptyset}$.
So points 1 and 2 from the lemma are a direct consequence of point 3.

The point 3 is proved in \cite[Lemma I.3.22]{houdebert_these}  or can be adapted from the proof of Lemma 3.5 in \cite{dereudre_2009}.

\subsubsection{Proof of Lemma \ref{lemme_preuve_dlr_localisation}}
Since $\MultiConf \in \A$, it is enough to check that balls  $(x,r) \in \MultiConf$ centred far enough can not overlap  $L(\MultiConf'_{\Lambda})$.
But since $\MultiConf'_{\Lambda} \in U_{R_1}$, we have 
$L(\MultiConf'_{\Lambda}) \subseteq \Lambda \oplus B(0,r_1)$.
Finally $\MultiConf \in \Upsilon_k$ and therefore we have for a set $\Delta$ large enough (which can be made explicit) that for $(x,r) \in \MultiConf_{\Delta^c}$, $B(x,r) \cap \Lambda \oplus B(0,R_1) = \emptyset$.
This concludes the proof of the Lemma.

\section{Proof of Theorem \ref{theo_polychromaticite_cas_extreme_faible activite} }
\label{section_preuve_non_monochromaticite_petite activite}
The first assertion is a consequence of the second which  we prove now.
Consider the sequence $(\MultiP_n)$ of finite-volume Widom-Rowlinson measures with free boundary condition, and $(\bar{\MultiP}_n)$ the stationary modification, defined in the proof of Theorem \ref{theo_existence_wr}, see Section \ref{section_preuve_theo_existence}.
As in the proof of Theorem \ref{theo_existence_wr}, the sequence $(\bar{\MultiP}_n)$ admits an accumulation point denoted by $\MultiP$, which satisfies $\MultiP (\A)=1$.
But the proof of the DLR equations done for the proof of Theorem \ref{theo_existence_wr} in Section \ref{section_preuve_theo_existence} is not valid in this case since Lemma \ref{lemme_preuve_dlr_ups_total} is false in the extreme case.
In fact to prove the DLR equations we will use the polychromaticity which we show now using the specific entropy.
\subsection{Lower bound of the specific entropy of monochromatic measures}
Consider a monochromatic stationary probability measure $\MultiP^{mono}$.
Without loss of generality consider $\MultiP^{mono}$ to be ergodic.
This implies in particular that the colour of $\MultiP^{mono}$ is deterministic, and let us call it $i$. 
Let us compute the specific entropy of $\MultiP^{mono}$.
Let $n \geq 0$.
If $\MultiP^{mono}_{\Lambda_n}$ is not absolutely continuous with respect to 
$\Poisson{\MultiParamPoisson}_{\Lambda_n}$, then 
$\Entropy_{\Lambda_n} 
(\MultiP^{mono} | \Poisson{\MultiParamPoisson}) = \infty$.
Otherwise 
$\MultiP^{mono}_{\Lambda_n} \ll \Poisson{\MultiParamPoisson}_{\Lambda_n}$ but, since $\MultiP^{mono}$ is monochromatic of colour $i$, then we also have
$\MultiP^{mono}_{\Lambda_n} \ll
\Poisson{\bar{\MultiIntensity}^i , \MultiRadius}_{\Lambda_n}$
where $\bar{\MultiIntensity}^i:= (0, \dots ,0, \Intensity_i, 0, \dots ,0)$
is the vector where the only non-zero coordinate being $z_i$ at the $i$-th position. Therefore
\begin{align}
\label{eq_minoration_entropy_mono_retour_poisson}
\Entropy_{\Lambda_n}  (\MultiP^{mono} | \Poisson{\MultiParamPoisson})
&= \nonumber
\Entropy_{\Lambda_n}  (\MultiP^{mono} | 
\Poisson{\bar{\MultiIntensity}^i , \MultiRadius})
+
\Entropy_{\Lambda_n}  (\Poisson{\bar{\MultiIntensity}^i , \MultiRadius}|
 \Poisson{\MultiParamPoisson})
\\ & \geq 
\Entropy_{\Lambda_n}  (\Poisson{\bar{\MultiIntensity}^i , \MultiRadius}|
 \Poisson{\MultiParamPoisson}),
\end{align}
where the inequality \eqref{eq_minoration_entropy_mono_retour_poisson} comes from the positivity of the relative entropy.
But a direct computation leads to
$$
\frac{d\Poisson{\bar{\MultiIntensity}^i , \MultiRadius}_{\Lambda_n}}
{d\Poisson{\MultiParamPoisson}_{\Lambda_n}} (\MultiConf)
= 
\prod_{j \not = i}  \exp (\Intensity_j \Leb(\Lambda_n) ) \1_{\Conf^j= \emptyset}
$$
which implies
\begin{align*}
\Entropy_{\Lambda_n}  (\Poisson{\bar{\MultiIntensity}^i , \MultiRadius}|
 \Poisson{\MultiParamPoisson})
 &= \Leb (\Lambda_n) \sum_{j \not = i} \Intensity_j
 \geq 
 \Leb (\Lambda_n) \min_i \left( \sum_{j \not = i} \Intensity_j \right).
\end{align*}
The last inequality together with \eqref{eq_minoration_entropy_mono_retour_poisson} implies
\begin{align}
\label{eq_minoration_entropy_monochromatique}
\Entropy(\MultiP^{mono})
\geq
\min_{i} \left( \sum_{j \not = i} \Intensity_j \right)
=
\Intensity \left( 1 - \max_i \Proportion_i \right).
\end{align}
\begin{remark}
Inequality \eqref{eq_minoration_entropy_monochromatique} cannot be improved since inequality \eqref{eq_minoration_entropy_mono_retour_poisson} becomes an equality in the case where $\MultiP^{mono} = 
\Poisson{\bar{\MultiIntensity}^i , \MultiRadius}$.
\end{remark}

\subsection{Upper bound of the specific entropy of $\MultiP$}
Let us now look at the specific entropy of $\MultiP$.
If we prove that $\Entropy(\MultiP)< 
\Intensity \left( 1 - \max_i \Proportion_i \right)$, then, by the lower bound \eqref{eq_minoration_entropy_monochromatique}, $\MultiP$ is not monochromatic.
The bound from Proposition \ref{propo_borne_entropi_spécifique_wr} is not good enough for the purpose of the current proof and we now improve it.
Divide the cube $\Lambda_n$ into copies of the smaller cube $\Lambda_m$ with $m<n$. We denote by $k_n$ the number of such copies of $\Lambda_m$ in $\Lambda_n$ and we fix $\epsilon<1- \Proportion_{\max}<1$.

One example of authorized configuration is when each copy of $\Lambda_m$ contains only one type of particles and when the balls do not overlap the outside of the cube.
This leads to the following inequality.
\begin{align}
\Poisson{\Intensity \MultiProportion , \MultiRadius}_{\Lambda_n}(\A)
\geq
\exp\left(-\Intensity \Leb (\Lambda_n)\right)
\left(  
1 + \sum_{ i} \sum_{l \geq 1}
\frac{\Intensity^l \Proportion_i^l \Leb (\Lambda_m)^l \phi_{m,i}^l}{l !}
\right)^{k_n},
\end{align}
with $\phi_{m,i} := \frac{1}{\Leb (\Lambda_m)} \int_{\R^d} \int_{\R^+} \1_{B(x,r) \subseteq \Lambda_m} Q_i(dr) dx$. 
Fix $\gamma <1 - \epsilon - \Proportion_{\max}$.
By choosing $m$ large enough, we have $\phi_{m,i} \geq 1 - \gamma$ for all $i$.
Furthermore we have
\begin{align}
\Entropy (\bar{\MultiP}_n)
&=  \nonumber
- \frac{\log(\Poisson{\Intensity \MultiProportion, \MultiRadius}_{\Lambda_n}(\A))}{\Leb (\Lambda_n)}
\\ & \leq
\Intensity  - \frac{k_n}{\Leb (\Lambda_n)} 
\log \left(1- \Types + \sum_i 
\exp(\Intensity \Proportion_i \Leb (\Lambda_m) \phi_{m,i}) 
\right).
\end{align}
Fix $\beta<1$ satisfying 
$\epsilon + \Proportion_{\max} \leq \beta (1-\gamma)$.
Using the fact that 
$\Leb (\Lambda_n)/k_n \underset{n \to \infty}{\longrightarrow} \Leb (\Lambda_m)$,
we have for $n$ large enough that
\begin{align}
\label{eq_preuve_non_mono_plus_de_n}
\Entropy (\bar{\MultiP}_n)
 \leq
\Intensity  - \frac{\beta}{\Leb (\Lambda_m)} 
\log \left(1- \Types + \sum_i 
\exp(\Intensity \Proportion_i \Leb (\Lambda_m) \phi_{m,i}) 
\right).
\end{align}
The bound from \eqref{eq_preuve_non_mono_plus_de_n} is not depending on $n$.
It remains to prove that the following function of $z$ is negative close to the origin, uniformly in $\MultiProportion$ satisfying \eqref{eq_condition_proportion},
\begin{align*}
\Psi_{\MultiProportion}(\Intensity)
:= 
\Intensity \max_i \Proportion_i - \frac{\beta}{\Leb (\Lambda_m)}
\log \left(1- \Types + \sum_i \exp(\Intensity \Proportion_i \Leb (\Lambda_m) \phi_{m,i}) \right).
\end{align*}
It satisfies $\Psi_{\MultiProportion}(0)=0$ and
\begin{align*}
\Psi'_{\MultiProportion}(\Intensity)
& =
\max_i \Proportion_i - \beta
\frac{\sum_i \Proportion_i \phi_{m,i}\exp(\Intensity \Proportion_i \Leb (\Lambda_m) \phi_{m,i})}
{1- \Types + \sum_i \exp(\Intensity \Proportion_i \Leb (\Lambda_m) \phi_{m,i})}
\\ & \leq
 \Proportion_{\max}
- \frac{\beta (1 - \gamma)}
{1- \Types + \sum_i \exp(\Intensity 
\Proportion_{max} \Leb (\Lambda_m) \phi_{m,i})},
\end{align*}
where the last bound does not depend on $\MultiProportion$.
Therefore by the choice of the parameters we have
$\Psi'_{\MultiProportion}(0)\leq - \epsilon$ and thus for $z$ small enough, uniform in $\MultiProportion$ satisfying \eqref{eq_condition_proportion}, we have 
$\Psi'_{\MultiProportion}(z)\leq -\frac{ \epsilon}{2}$.
Therefore using the lower semi-continuity of the specific entropy, we get that $\Entropy (\MultiP) < 
\Intensity \left( 1 - \max_i \Proportion_i \right)$, which implies together with \eqref{eq_minoration_entropy_monochromatique} that $\MultiP(Poly)>0$.

Now by conditioning on the event $Poly$, one can prove that the conditioned probability measure satisfies the DLR equations \eqref{eq_dlr_definition_wr}.
This was done in detail in \cite{dereudre_houdebert} for the symmetric case.
%\begin{remark}
%The key element of the proof is to show that every accumulation point of the sequence $(\bar{\MultiP}_n)$ is not monochromatic for small activities.
%However the proof does not give any clue what is happening for large activities.
%\end{remark}

\section{Proof of Theorem \ref{theo_conjecture_preuve_faible}}
\label{section_preuve_conjecture_faible}
\subsection{Fortuin-Kasteleyn representation}
We consider the sequence $(\bar{\MultiP}_n^{\text{free}})$ defined in Section \ref{section_results} which admits, 
thanks to Proposition \ref{propo_compactness_entropy_specifique}, at least one cluster point denoted $\MultiP^{\text{free}}$. 
The second part of the theorem, which claims that, for small activity, $\MultiP^{\text{free}}$ is polychromatic, is similar to the proof of Theorem \ref{theo_polychromaticite_cas_extreme_faible activite}. 
We do not give the details here. So it remains to prove the first part of the theorem on the monochromaticity. 

In a same fashion as Section \ref{section_preuve_theo_existence}, one can easily prove that $\MultiP^{\text{free}}(\A)=1$.
So one way of proving that $\MultiP^{\text{free}}$ is monochromatic is to prove that $\MultiP^{\text{free}}$ covers the whole space $\R^\Dim$ with one giant connected component.

The first step of the proof is to transfer the problem from multi-type Widom-Rowlinson measure in $\MultiConfSpace$ to "single-type" Continuum Random Cluster measure in $\ConfSpace$.

To a measure $\MultiP$ on $\MultiConfSpace$ we associate its \emph{color-blind} measure, denoted by $P$, which is a probability measure on $\ConfSpace$ defined by
\begin{align*}
\forall E \in \SigmaAlgebra, \ 
P(E) = \MultiP (\{\MultiConf \, | \cup_i \Conf^i \in E  \}) \,.
\end{align*}

The specific entropy and the local convergence topology can be defined the same as for $\MultiConfSpace$.
In particular $P^{\text{free}}$ is still an accumulation point of 
the sequence $(\bar{P}_n^{\text{free}})$.
\begin{proposition}[Fortuin-Kasteleyn representation]
The measure $P_n^{\text{free}}$ is a \emph{Continuum Random Cluster measure} on the bounded box $\Lambda_n^{(k_n)}$ with free boundary condition and with parameters $\Types, \Intensity$ and $\Radius$.
This means
\begin{align*}
P_n^{\text{free}}(d\Conf)= 
\frac{\Types^{\ncc (\Conf)}}{Z_n^{(k_n)}} 
\Poisson{\Intensity,\Radius}_{\Lambda_n^{(k_n)}}(d\Conf),
\end{align*}
where $\ncc (\Conf)$ denotes the number of connected components of the structure $L(\Conf)$, and where $Z_n^{(k_n)}$ is the normalizing constant.
\end{proposition}
This proposition is a standard result, known as the \emph{Fortuin-Kasteleyn representation} or \emph{grey representation}, proved and used in a lot of articles such as \cite{chayes_kotecky,georgii_haggstrom,dereudre_houdebert,houdebert_2017}.

The next proposition is the key element of the proof of Theorem \ref{theo_conjecture_preuve_faible}.
\begin{proposition}
\label{proposition_majoration_ncc}
In the assumptions of Theorem \ref{theo_conjecture_preuve_faible}, there exists $\tilde{\Intensity}<\infty$ such that for all 
$\Intensity>\tilde{\Intensity} $ and for all $n$

\begin{align*}
\int_{\ConfSpace} \ncc (\Conf) P_n^{\text{free}}(d\omega) \leq D \times C^{k_n},
\end{align*}
where $C,D>1$ are finite constants not depending on $n$.
\end{proposition}
\begin{remark}
In the case $\Dim=1$ one can prove in the same fashion the bound
\begin{align*}
\int_{\ConfSpace} \ncc (\Conf) P_n^{\text{free}}(d\omega) \leq
C
\end{align*}
for large enough $z$ and for a constant $C>1$ independent of $n$.
\end{remark}
Before proving this proposition, let us show how it leads to the expected result.
By the computation done for the proof of Theorem \ref{theo_existence_wr}, see Section \ref{section_preuve_theo_existence}, we have
\begin{align*}
\Entropy (\bar{P}_n^{\text{free}})
&= \frac{1}{\Leb \left(\Lambda_n^{(k_n)} \right)}
\Entropy \left(P_n^{\text{free}} | 
\Poisson{\Intensity, \Radius}_{\Lambda_n^{(k_n)}} \right)
\\ &= 
\frac{1}{\Leb \left(\Lambda_n^{(k_n)} \right)}
\left(
-\log (Z_n^{(k_n)}) +  \log(\Types) 
\int_\ConfSpace \ncc \ d P_n^{\text{free}}
\right)
\\ &\leq \frac{D \times C^{k_n} \log (\Types)}{2nk_n^{d-1}}\\
& \leq D \log(q) C^{k_n-\log(n)/\log(C)}.  
\end{align*}
Since $(k_n)$ is negligible with respect to $\log(n)$, this upper-bound tends to zero when $n$ goes to infinty. So by the lower semi-continuity of the specific entropy, we have
$\Entropy (P^{\text{free}})=0$ which implies that $P^{\text{free}}=\Poisson{\Intensity, \Radius}$.
Noting that the condition \eqref{eq_conjecture_cond1} implies in particular that $\int_{\R^+}r^\Dim \Radius (dr) = \infty$, so the Poisson Boolean model covers almost surely the whole space $\R^d$ with one connected component.
Therefore, since $\MultiP^{\text{free}}(\A)=1$, the measure $\MultiP^{\text{free}}$ has no choice but to be monochromatic. By symmetry of the model, each color appears with probability $1/q$.

\subsection{Proof of Proposition \ref{proposition_majoration_ncc}}
First let $k>0$ which is taken for simplicity as the inverse of a positive integer.
A condition on $k$ will appear later. 
Furthermore let us consider $\Typesbis>\Types$ such that $\Radius (\{ 0 \})<1/\Typesbis$.

Since $Z_n^{(k_n)}\geq 1$, we have for a constant $D>1$ large enough, 
\begin{align}
\label{eq_preuve_conj_retour_poisson}
\int_{\ConfSpace} \ncc \  d  P_n^{\text{free}}  
& \leq 
\int_{\ConfSpace}  \ncc \ \Types^{ \ncc } \ d
\Poisson{\Intensity , \Radius }_{\Lambda_n^{(k_n)}}
\nonumber \\ & \leq
D \int_{\ConfSpace}  \Typesbis^{ \ncc } \ d
\Poisson{\Intensity , \Radius }_{\Lambda_n^{(k_n)}}
\nonumber \\ & \leq
D \left( \int_{\ConfSpace}  \Typesbis^{ \ncc } \ d
\Poisson{\Intensity , \Radius }_{\Lambda_n^{(k)}} \right) ^{k_n/k},
\end{align}
where the last inequality is a consequence of the stationarity of the Poisson Boolean model and the subadditivity of the number of connected components, i.e.
$\ncc(\Conf_{\Delta_1 \cap \Delta_2})
\leq \ncc(\Conf_{\Delta_1 }) +\ncc(\Conf_{ \Delta_2})$ for disjoint $\Delta_1$ and $\Delta_2$.
\begin{remark}
In the case $\Dim=1$ there is no need to introduce $k_n$ and $k$ and we directly have
\begin{align*}
\int_{\ConfSpace} \ncc \  d  P_n^{\text{free}}  
\leq D
\int_{\ConfSpace}  \Typesbis^{ \ncc } \ d
\Poisson{\Intensity , \Radius }_{\Lambda_n},
\end{align*}
where $\Lambda_n=]-n,n]$.
\end{remark}

Therefore, it is enough to show that that for a fixed $k$ small enough the sequence
\begin{align*}
  \int_{\ConfSpace}  \Typesbis^{ \ncc } \ d
\Poisson{\Intensity , \Radius }_{\Lambda_n^{(k)}}
\end{align*}
is uniformly bounded in $n\ge 1$.

We are going to introduce the notion of \emph{connected component to the right}.
Thanks to a renewal phenomenon, the number of connected components to the right will have a geometric law.
Then we will prove, using the geometric law, that this number of connected components to the right admits exponential moments for $\Intensity$ large enough.

More precisely we have
\begin{align}
\label{eq_preuve_conj_droite_1}
\int_{\ConfSpace} \Typesbis^{ \ncc } \ d
\Poisson{\Intensity , \Radius }_{\Lambda_n^{(k)}}
 & \leq
\left(
\int_{\ConfSpace}  \Typesbis^{ \ncc } \ d
\Poisson{\Intensity , \Radius }_{\Lambda_n^{(k+)}} \right)^2
\end{align}
where  this inequality comes from the stationarity and the independence property of $\Poisson{\Intensity, \Radius}$.

Now the problem is that when $n$ grows, a ball can appear and cover all existing connected components.
So seeing $n$ as the time, the present is depending on the yet unknown future.
To overcome this issue let us do the following.
To a radius $r$ we associate the following transformed radius
\begin{align*}
\tilde{r}
=
\sqrt{r^2 - (\Dim -1)k^2} \ \1_{r^2 \geq (\Dim -1)k^2}.
\end{align*}
This way if $(x,r)$ is a marked point, the set 
\begin{align*}
T(x,\tilde{r})=[x,x+\tilde{r}]\times [0,k]^{\Dim - 1}
\end{align*} 
does not cover anything left of $x$ and is entirely covered by the ball $B(x,r)$.
%, see Figure \ref{fig_nccd}.
\begin{remark}
In dimension $\Dim=1$ we simply consider $\tilde{r}=r$.
\end{remark}
Now we are defining the \emph{number of connected components to the right}, denoted by $ \nccd (\Conf)$, as the number of connected components of
$\underset{(x,r) \in \Conf}{\bigcup} T(x,\tilde{r})$.
%\begin{figure}
%\begin{center}
%\includegraphics[width=10cm]{cc_droite.eps}
%\end{center}
%\caption{blop}
%\label{fig_nccd}
%\end{figure}
We have $\ncc \leq \nccd $ and $\nccd$ is increasing with respect to $n$, which implies
\begin{align}
\label{eq_preuve_conj_plus_de_n_1}
\int_{\ConfSpace}  \Typesbis^{
\ncc } \ d 
\Poisson{\Intensity, \Radius}_{\Lambda_n^{(k+)}}
 &\leq
\int_{\ConfSpace}  \ \Typesbis^{
\nccd } \ d 
\Poisson{\Intensity, \Radius}_{\Lambda_n^{(k+)}}
\leq
\int_{\ConfSpace}   \Typesbis^{\nccd } \ d
\Poisson{\Intensity, \Radius}_{\Lambda_\infty^{(k+)}}
\end{align}
where the right-hand side in \eqref{eq_preuve_conj_plus_de_n_1} does not depend on $n$.
Therefore in order to prove Proposition \ref{proposition_majoration_ncc} we have to prove that the right-hand side of \eqref{eq_preuve_conj_plus_de_n_1} is finite.

The first thing to do is to prove that the quantity 
$\nccd \left(\Conf \right)$ is 
$\Poisson{\Intensity, \Radius}_{\Lambda_{\infty}^{(k+)}}$ almost surely finite.
Then we will be able to show that this quantity admits exponential moments for $\Intensity$ large enough. To this end it is sufficient to study the number of connected components of a segment model where the left extremities of the segments are the points of a Poisson point process (on the real line) of intensity  $\Intensity k^\Dim$,
 and where the lengths of the segments are independent random variables distributed accordingly to the measure 
$\tilde{\Radius}$ satisfying 
\begin{align*}
\tilde{\Radius} \left([0,r]\right) = 
\Radius (\, [ 0 , \sqrt{r^2 + (d-1)k^2} ] \ ) .
\end{align*}
Condition \eqref{eq_conjecture_cond1} transfers trivially to $\tilde{\Radius}$.
By taking $k$ small enough, we also have $\tilde{\Radius}(\{ 0 \}) <1/\Typesbis$.

The next lemma, proved in \cite[Corollary 4]{fitzsimmons_fristedt_shepp}, ensures that
$\nccd \left(\Conf\right)$ is 
$\Poisson{\Intensity, \Radius}_{\Lambda_{\infty}^{(k+)}}$ almost surely finite.
\begin{lemma}
\label{lemme_percolation_segment_a_droite}
For every $\Intensity \geq k^{-\Dim}$, the quantity
$\nccd \left(\Conf\right)$ is 
$\Poisson{\Intensity, \Radius}_{\Lambda_{\infty}^{(k+)}}$ almost surely finite.
\end{lemma}
This lemma, working thanks to the condition \eqref{eq_conjecture_cond1}, ensures in particular that up to a random point $\mathcal{Y}$, the "cylinder" 
$[\mathcal{Y},\infty] \times [0,k]^{\Dim -1}$ is covered by 
$\underset{(x,r) \in \Conf_{\Lambda_{\infty}^{(k+)}} }{\bigcup}
T(x,\tilde{r})$.
So when we go through $\Lambda_{\infty}^{(k+)}$ from the origin, each connected component encountered has a positive probability $p(z)$ of being the infinite one, independently of all finite connected components already encountered.
Therefore the random variable 
$\nccd \left(\Conf \right)$ is geometric of mean $p(z)^{-1}$ satisfying
\begin{align*}
p(z)^{-1}= \int_{\ConfSpace} \nccd 
\left(\Conf \right)
\Poisson{\Intensity, \Radius}_{\Lambda_{\infty}^{(k+)}} (d \Conf)
\end{align*}
and
\begin{align*}
p(z)= \Poisson{\Intensity, \Radius}_{\Lambda_{\infty}^{(k+)}}
\left(  \nccd =1 \right).
\end{align*}
So an explicit computation of the right-hand side of \eqref{eq_preuve_conj_plus_de_n_1} leads to
\begin{align}
\label{eq_somme_geometrique}
\int_{\ConfSpace}  \Typesbis^{\nccd} \ d
\Poisson{\Intensity, \Radius}_{\Lambda_{\infty}^{(k+)}}
=
\sum_{\alpha \geq 0} \Typesbis^\alpha p(z) (1-p(z))^{\alpha -1}
\end{align}
which is finite as soon as $p(z)> 1 - 1/ \Typesbis$.

In the following lemma we prove that for large activities $\Intensity$, $p(z)$ is as close to 1 as we need.
\begin{lemma}
\label{lemme_preuve_conv_geometrique}
For  $\Intensity$ large enough
\begin{align*}
p(z)  >  1 - 1/ \Typesbis .
\end{align*}
\end{lemma}
With this lemma we have for $\Intensity$ large enough that
\begin{align}
C := \left( \int_{\ConfSpace}  
\Typesbis^{\nccd } \ d
\Poisson{\Intensity, \Radius}_{\Lambda_{\infty}^{(k+)}} \right) ^{2/k}
< \infty   \nonumber
 \end{align}
 which combined with \eqref{eq_preuve_conj_retour_poisson},
\eqref{eq_preuve_conj_droite_1}, 
\eqref{eq_preuve_conj_plus_de_n_1}
and \eqref{eq_somme_geometrique} concludes the proof of Proposition \ref{proposition_majoration_ncc}.

The last thing needed is to prove Lemma \ref{lemme_preuve_conv_geometrique}.

\subsection{Proof of Lemma \ref{lemme_preuve_conv_geometrique}}
We know that $\nccd \left(\Conf\right)$ is a geometric random variable.
We have
\begin{align}
p(z) & =
\Poisson{\Intensity, \Radius}_{\Lambda_{\infty}^{(k+)}} 
\left(  \nccd =1 \right)
\nonumber \\ & \geq 
\Poisson{\Intensity, \Radius}_{\Lambda_{\infty}^{(k+)}}  \Big( 
\nccd \left(\Conf\right) =1,
[x + y,\infty[ \text{ is right-covered in }  
\Conf_{]x, \infty[ \times[0,k]^{\Dim -1}} \Big)
,   \nonumber
\end{align}
where $x$ and $y$ are positive real numbers and 
we say that $[x,\infty[$ is right-covered in $\Conf$ if 
\begin{align}
[x,\infty[ \times[0,k]^{\Dim -1}  \ \subseteq
\underset{ (a,r) \in \Conf}{\bigcup} T(a,\tilde{r}).
\nonumber
\end{align}
Therefore we have
\begin{align}
p(z)
 \geq &
\left(1-e^{-\Intensity x} \right) \
\tilde{\Radius}([x+y, \infty [) 
\nonumber \\
& \hspace{3cm}
\Poisson{\Intensity, \Radius} ( [x + y,\infty[ \text{ is right-covered in }  
\Conf_{]x, \infty[ \times[0,k]^{\Dim -1}}) 
   \nonumber
\\ =&
\left(1-e^{-\Intensity x  }\right) \
\tilde{\Radius}([x+y, \infty [) 
\nonumber \\
& \hspace{3cm} 
\Poisson{\Intensity, \Radius} ( [y,\infty[ \text{ is right-covered in }  
\Conf_{[0, \infty[ \times[0,k]^{\Dim -1}}),
\nonumber
\end{align}
where the last equality comes from the stationarity of $\Poisson{\Intensity, \Radius}$.

\begin{lemma}
\label{lemme_convergence_point_le_plus_a_gauche}
For $y>0$ fixed, the probability
\begin{align}
\Poisson{\Intensity, \Radius} ( [y,\infty[ \text{ is right-covered in }  
\Conf_{[0, \infty[ \times[0,k]^{\Dim -1}})
\nonumber
\end{align}
converges to $1$ when $z$ grows to infinity.
\end{lemma}

\begin{proof}[Proof of Lemma \ref{lemme_convergence_point_le_plus_a_gauche}]
Consider the realisation of a Poisson point process on the half line of intensity $\Intensity $ and with segment length of law 
$\tilde{\Radius}$. 
Thanks to Lemma \ref{lemme_percolation_segment_a_droite},
we know that this segment model percolates for $\Intensity\geq k ^{- \Dim}$.
Therefore it exists a infinite connected component starting from 
a random point $\mathcal{Y}$.
If $y$ is inside this infinite connected component then there is nothing to do.
Otherwise by increasing  $\Intensity$ we are adding Poisson point and 
$\mathcal{Y}$ is translating to the left, converging almost surely to $0$ when $z$ grows to infinity, and overlapping $y$ after a finite random time.
From this almost surely convergence we deduce the convergence of the probability of Lemma \ref{lemme_convergence_point_le_plus_a_gauche}.
\end{proof}

Let us now conclude the proof of Lemma \ref{lemme_preuve_conv_geometrique}.

For  $x,y$ small enough, thanks to the construction of $\tilde{\Radius}$ satisfying  $\tilde{\Radius}(\{0 \})<1/ \Typesbis$, we have
$\tilde{\Radius}([x+y, \infty [) > 1-1 / \Typesbis$.
Now that $x$ and $y$ are fixed, thanks to Lemma \ref{lemme_convergence_point_le_plus_a_gauche} we have for $z$ large enough that the same is true for
\begin{align}
\left(1-e^{-\Intensity x  }\right) \
\tilde{\Radius}([x+y, \infty [) 
\Poisson{\Intensity, \Radius} ( [y,\infty[ \text{ is right-covered in }  
\Conf_{[0, \infty[ \times[0,k]^{\Dim -1}}) .
\nonumber
\end{align}
Thus for $\Intensity$ large enough $p(\Intensity)>1-1/ \Typesbis$ and Lemma \ref{lemme_preuve_conv_geometrique} is proved.

\vspace{1cm}
%%%%%%%%%%%%%%%%%%%%%%%%%%%
{\it Acknowledgement:} This work was supported in part by the Labex CEMPI  (ANR-11-LABX-0007-01), the GDR 3477 Geosto and the ANP project PPPP (ANR-16-CE40-0016).

\newpage 
%%%%%%%
%\nocite{*}
\bibliographystyle{plain}
\bibliography{biblio}
\end{document}